\documentclass[12pt,reqno]{amsart}

\usepackage{times}
\usepackage[T1]{fontenc}
\usepackage{mathrsfs}
\usepackage{latexsym}
\usepackage[dvips]{graphics}
\usepackage[dvips]{graphicx}
\usepackage{epsfig}
\usepackage{amsmath,amsfonts,amsthm,amssymb,amscd}
\input amssym.def
\input amssym.tex
\usepackage{color}
\usepackage{nicefrac}

\newcommand{\comment}[1]{}

\newcommand\be{\begin{equation}}
\newcommand\ee{\end{equation}}
\newcommand\bea{\begin{eqnarray}}
\newcommand\eea{\end{eqnarray}}
\newcommand\nbea{\begin{eqnarray*}}
\newcommand\neea{\end{eqnarray*}}
\newcommand\bi{\begin{itemize}}
\newcommand\ei{\end{itemize}}
\newcommand\ben{\begin{enumerate}}
\newcommand\een{\end{enumerate}}

\newtheorem{thm}{Theorem}[section]

\newtheorem{cor}[thm]{Corollary}
\newtheorem{lem}[thm]{Lemma}

\newtheorem{rek}[thm]{Remark}

\numberwithin{equation}{section}

\newcommand{\bal}{\begin{align}}
\newcommand{\eal}{\end{align}}

\newcommand{\nn}{\nonumber\\}

\begin{document}

\title[Generalized More Sums Than Differences Sets]{Generalized More Sums Than Differences Sets}

\author{Geoffrey Iyer}\email{geoff.iyer@gmail.com}
\address{Department of Mathematics, University of Michigan, Ann Arbor, MI 48109}

\author{Oleg Lazarev}\email{olazarev@Princeton.edu}
\address{Department of Mathematics, Princeton University, Princeton, NJ 08544}

\author{Steven J. Miller}\email{sjm1@williams.edu, Steven.Miller.MC.96@aya.yale.edu}
\address{Department of Mathematics and Statistics, Williams College,
Williamstown, MA 01267}

\author{Liyang Zhang}\email{lz1@williams.edu}
\address{Department of Mathematics and Statistics, Williams College,
Williamstown, MA 01267}



\date{\today}

\thanks{We thank the participants of various CANT Conferences (especially Kevin O'Bryant, Greg Martin, Mel Nathanson and Jonathan Sondow), of the 2011 Young Mathematicians Conference at Ohio State, and the SMALL 2011 REU at Williams College for many enlightening conversations, and the reviewer for comments on the manuscript. The first, second and fourth named authors were supported by NSF grants DMS0850577 and Williams College; the third named author was partially supported by NSF grant DMS0970067.}

\begin{abstract} A More Sums Than Differences (MSTD, or sum-dominant) set is a finite set $A\subset \mathbb{Z}$ such that $|A+A|<|A-A|$. Though it was believed that the percentage of subsets of $\{0,\dots,n\}$ that are sum-dominant tends to zero, in 2006 Martin and O'Bryant \cite{MO} proved that a positive percentage are sum-dominant. We generalize their result to the many different ways of taking sums and differences of a set. We prove that $|\epsilon_1A+\cdots+\epsilon_kA|>|\delta_1A+\cdots+\delta_kA|$ a positive percent of the time for all nontrivial choices of $\epsilon_j,\delta_j\in \{-1,1\}$. Previous approaches proved the existence of infinitely many such sets given the existence of one; however, no method existed to construct such a set. We develop a new, explicit construction for one such set, and then extend to a positive percentage of sets.

We extend these results further, finding sets that exhibit different behavior as more sums/differences are taken. For example, we prove that for any $m$, $|\epsilon_1A + \cdots + \epsilon_kA| - |\delta_1A + \cdots + \delta_kA| = m$ a positive percentage of the time. We find the limiting behavior of $kA=A+\cdots+A$ for an arbitrary set $A$ as $k\to\infty$ and an upper bound of $k$ for such behavior to settle down. Finally, we say $A$ is $k$-generational sum-dominant if $A$, $A+A$, $\dots$, $kA$ are all sum-dominant. Numerical searches were unable to find even a 2-generational set (heuristics indicate that the probability is at most $10^{-9}$, and quite likely significantly less). We prove that for any $k$ a positive percentage of sets are $k$-generational, and no set can be $k$-generational for all $k$.
\end{abstract}

\maketitle

\tableofcontents


\section{Introduction}

Given a finite set of integers $A$, two natural sets to study are \bea A+A & \ = \ & \{a_1 + a_2: a_1, a_2 \in A\} \nonumber\\ A-A &=& \{a_1 - a_2: a_1, a_2 \in A\}.\eea The most natural question to ask is: As we vary $A$ over a family of sets, how often is $|A+A| > |A-A|$ (where $|X|$ is the cardinality of $X$)? We call such sets More Sums Than Differences (MSTD) sets, or sum-dominant (if the two cardinalities are the same we say $A$ is balanced, and if $|A-A|>|A-A|$ we say $A$ is difference-dominant). As addition is commutative but subtraction is not, a typical pair contributes two differences to $A-A$ but only one sum to $A+A$. While there are numerous constructions of such sets and infinite families of such sets \cite{He,HM2,Ma,MOS,Na2,Na3,Na4,Ru1,Ru2,Ru3}, one expects sum-dominant sets to be rare; however, Martin and O'Bryant \cite{MO} proved that a positive percentage of sets are sum-dominant. They showed the percentage is at least $2 \cdot 10^{-7}$, which was improved by Zhao \cite{Zh2} to at least $4.28 \cdot 10^{-4}$ (Monte Carlo simulations suggest the true answer is about $4.5 \cdot 10^{-4}$). In all these arguments, each integer in $\{0,\dots,n-1\}$ has an equal chance of being in $A$ or not being in $A$, and thus all of the $2^n$ subsets are equally likely to be chosen. The situation is dramatically different if we consider a binomial model where the probability parameter tends to zero. Explicitly, for each $n$ let $p(n) \in (0,1)$. Now assume each integer in $\{0,\dots,n-1\}$ is chosen with probability $p(n)$. If $p(n)$ decays to zero with $n$, then Hegarty and Miller \cite{HM1} proved that with probability tending to 1 a randomly chosen set is difference-dominated. See \cite{ILMZ} for a survey of results in the field.\\

\emph{Throughout this paper we use the following notations:}
\begin{itemize}
\item $m\cdot A = \{m\cdot a: \, a\in A\}$. \\
\item $A + B = \{a+b : a \in A, b \in B\}$, $A - B = \{a-b: a \in A, b \in B\}$.\\
\item $|A|$ is the number of elements in $A$.\\
\item $mA = \underbrace{A+\cdots+A}_{m\text{ times}}$ if $m \ge 1$ (if $m=0$ we define $0A$ to be the empty set).\\
\item $-A = \{-a: a \in A\}$, and if $m \ge 0$ then $-mA = -(mA)$; note that if $m, n \ge 0$ then $mA + nA = (m+n)A$; however, $mA - nA \neq (m-n)A$.\\
\item $[a,b] = \{a,a+1,\ldots,b-1,b\}$.
\end{itemize}
\ \\

The purpose of this article is to generalize the positive percentage and explicit constructions of MSTD sets. Two natural questions, which motivated much of this work, are

\begin{enumerate}

\item Given non-negative integers $s_1, d_1, s_2, d_2$ with $s_1+d_1 = s_2+d_2 \ge 2$, can we find a set $A$ with $|s_1A-d_1A| > |s_2A-d_2A|$, and if so, does this occur a positive percentage of the time?

\item We say a set is $k$-generational if $A$, $A+A$, $\dots$, $kA$ are all sum-dominant. Do $k$-generational sets exist, and if so, do they occur a positive percentage of the time? Is there a set that is $k$-generational for all $k$?

\end{enumerate}

The answer to the first question is yes, and in fact the result can be generalized. When $s_1+d_1 = 2$, the only possible sets are essentially $A+A$ and $A-A$, as $-A-A$ is just the negation of $A+A$. When $s_1+d_1=3$, again there are essentially just two possibilities, $A+A+A$ and $A+A-A$, since $A-A-A = -(A+A-A)$ and thus we might as well assume $s_i \ge d_i$. New behavior emerges once the sum is at least 4. In that case, we have $A+A+A+A$, $A+A+A-A$ and $A+A-A-A$. One of our main results is that all possible orderings of these three sets happen a positive percentage of the time. This generalizes and improves results from \cite{MOS}, where large families were found with $|A+A+A| > |A+A-A|$ and large families could be found for more general binary comparisons \emph{if} one such set could be found.

For the second question, brute force numerical explorations could not find such sets. This is not surprising, as such sets are expected to be rare (simple heuristics imply that the percentage of such sets is at most $10^{-9}$, and quite likely much less). Generalizing our construction for the first problem, we find a positive percentage of sets are $k$-generational for any $k$; further, no set can be $k$-generational for all $k$.

We now state our main results and give a sketch of the proofs.

\begin{thm}\label{thm:generalizedMSTD} Let $s_1,d_1,s_2,d_2$ be non-negative integers such that $\{s_1,d_1\} \neq \{s_2,d_2\}$.
\begin{enumerate}
\item There exists a finite, non-empty set $A$ of non-negative integers such that $\left|s_1A-d_1A\right| > \left|s_2A-d_2A\right|$.
\item A positive percentage of finite subsets $A$ of non-negative integers satisfy $\left|s_1A-d_1A\right| > \left|s_2A-d_2A\right|$. Explicitly, there is a constant $c(s_1,d_1,s_2,d_2) > 0$ such that the number of subsets $A$ of $\{0,1,\dots,n-1\}$ satisfying $\left|s_1A-d_1A\right| > \left|s_2A-d_2A\right|$ is at least $c(s_1,d_1,s_2,d_2) 2^n$ as $n\to\infty$.
\end{enumerate}
\end{thm}

\begin{rek}\label{rek:sketchmainproof}\emph{Sketch of the proof:} The difficulty is finding one such set; after such a set has been found, we can modify the method of Martin and O'Bryant \cite{MO} to obtain a positive percentage. To create such a set $A$, we decompose $A$ into its left and right parts, denoted $L$ and $R$. We pick $L$ and $R$ to be almost symmetric, but we have $R$ slightly longer than $L$. Next, note that the left (resp. right) fringe of $xA-yA$ is given by $xL-yR$ (resp. $yL-xR$). Because of the near-symmetry of $L$ and $R$, the fringes of $xA-yA$ will have similar structure for different values of $x,y$. However, because $R$ is longer than $L$, the total length of a fringe depends on the number of copies of $R,L$.

\begin{figure}[h!]
\centering
\includegraphics[width=1\textwidth]{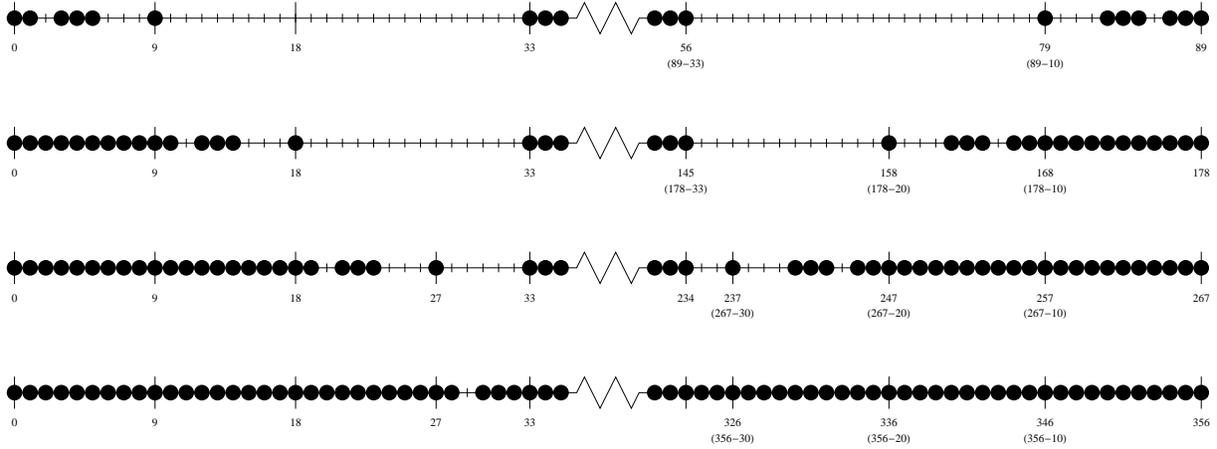}
\caption{\label{fig:set4s} $A$, $A+A$, $A+A+A$, and $A+A+A+A$. The sawtooth means all elements are present in that range.}
\end{figure}

In Figures \ref{fig:set4s} and \ref{fig:set2s2d}, we exhibit a set $A$ where $|2A+2A|>|2A-2A|$. Figure \ref{fig:set4s} shows $A+A+A+A$, while Figure \ref{fig:set2s2d} shows $A+A-A-A$.
Notice that in $A+A+A+A$, the right fringe intersects with the middle, which fills in all the gaps. The left fringe, on the other hand, grows too slowly to completely intersect with the middle, and is left with one gap.

\begin{figure}[h!]
\centering
\includegraphics[width=1\textwidth]{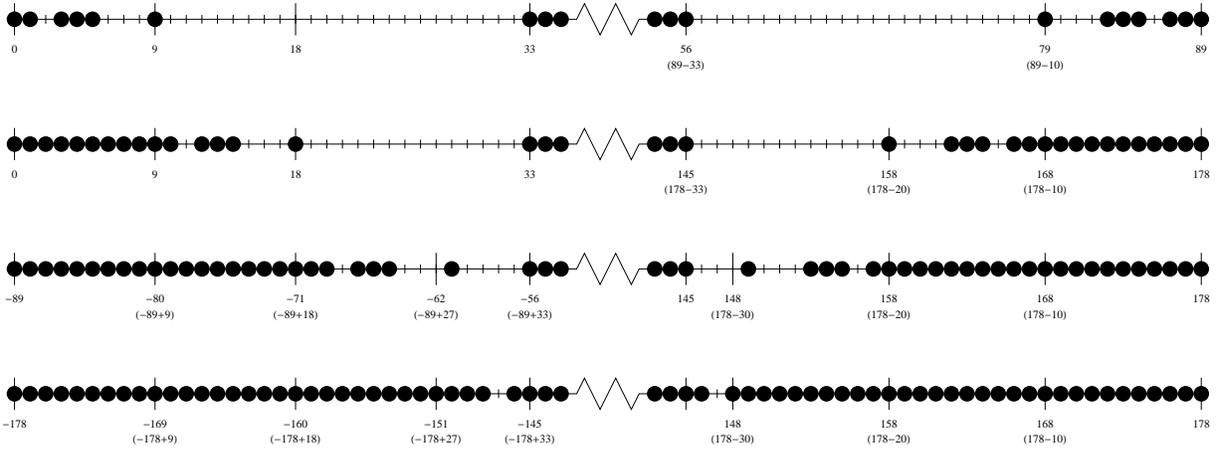}
\caption{\label{fig:set2s2d} $A$, $A+A$, $A+A-A$, and $A+A-A-A$.}
\end{figure}

In $A+A-A-A$, the left fringe is given by $L+L-R-R$, which has a length between $L+L+L+L$ and $R+R+R+R$. This is not quite long enough to intersect with the middle. Similarly, the right fringe is given by $R+R-L-L$, which is once again too short. Therefore $A+A-A-A$ is missing two elements.
\end{rek}

Theorem \ref{thm:generalizedMSTD} can be generalized to obtain the following.

\begin{thm}[Arbitrary Differences]\label{thm:arbitrarydifference}
Let  $a,b,c,d$ be non-negative integers such that $a > b,c,d$ and $a+b = c+d = q$. If $c\ne d$, then for any non-negative integers $m, \ell$ such that $\ell \leq 2m$ and all sufficiently large $n$, there exists $A\subseteq [0,n]$ such that $|aA-bA| = qn+1 - m$ and $|cA-dA|= qn+1-\ell$. If $c=d$, then the statement holds with the additional condition that $\ell$ is even.
\end{thm}

The next theorem constructs chains of Generalized MSTD sets. We start at $k=2$ below as there is essentially only one possibility when $k=1$ (namely the sets $A$ and $-A$, which must have the same cardinality).

\begin{thm}[Chains of Generalized MSTD Sets]\label{thm:xjyjwjzj} Let $x_{j}, y_{j}, w_{j}, z_{j}$ be finite sequences of non-negative integers of length $k$ such that $x_{j}+y_{j}=w_{j}+z_{j}=j$, and $\{x_j,y_j\} \neq \{w_j,z_j\}$ for every $2\leq j\leq k$. A positive percentage of sets $A$ satisfy $\left|x_{j}A-y_{j}A\right|>\left|w_{j}A-z_{j}A\right|$ for every $2\leq j\leq k$.
\end{thm}

\begin{thm}[Simultaneous Comparisons]\label{thm:simulcomparison} Given finite, non-negative sequences of length $n\leq\left\lfloor \frac{k}{2}\right\rfloor +1$
called $s_{j},d_{j}$ such that $s_{j}+d_{j}=k$ for all $1\leq j\leq k$
and $\{s_j,d_j\} \neq  \{s_i,d_i\}$ whenever $j\neq i$,
there exists a set $A$ such that $\left|s_{n}A-d_{n}A\right|>\cdots>\left|s_{1}A-d_{1}A\right|$\end{thm}

\begin{rek} The bound $n\leq\left\lfloor \frac{k}{2}\right\rfloor +1$ in the above theorem is completely artificial, as the condition $\{s_j,d_j\} \neq  \{s_i,d_i\}$ is impossible for $n > \left\lfloor \frac{k}{2}\right\rfloor +1$.
\end{rek}

\begin{rek} It is possible to combine Theorems \ref{thm:xjyjwjzj} and \ref{thm:simulcomparison} to obtain a set $A$ that satisfies the criteria in Theorem \ref{thm:simulcomparison} over many iterations of sums/differences.
\end{rek}

From Theorem \ref{thm:xjyjwjzj} we deduce

\begin{cor}[$k$-Generational Sets]\label{cor:kgenerationalexists}\

\begin{enumerate}

\item For each $k$, there exists a $k$-generational set. That
is, for each $k$, there exists a set $A$ such that $\left|cA+cA\right|>\left|cA-cA\right|$
for all $1\leq c\leq k$.

\item For each $k$, a positive percentage of sets are $k$-generational.

\item There is no set which is $k$-generational for all $k$.

\end{enumerate}

\end{cor}

The paper is organized as follows. In \S\ref{sec:proofthmgeneralizedMSTD} we explicitly construct one set with the properties in Theorem \ref{thm:generalizedMSTD}, obtaining only existence and not a positive percentage. For completeness we provide most of the verifications; the reader willing to accept their existence can move on to \S\ref{sec:pospercent}, where we generalize the method of Martin and O'Bryant to improve our results from the existence of one set to a positive percentage, completing the proof of Theorem \ref{thm:generalizedMSTD}. As the proof of arbitrary differences (Theorem \ref{thm:arbitrarydifference}) is not needed for the remaining results and is somewhat long and technical, we give it in Appendix \ref{sec:arbitrarydifference}.

Section \ref{sec:technicalconstructions} contains a few lemmas required to construct the sets in Theorems \ref{thm:xjyjwjzj} and \ref{thm:simulcomparison}. Once again, a reader uninterested in technical constructions may skip this section and proceed to \S\ref{sec:kgenerational}. We discuss $k$-generational sets (and related problems) in \S\ref{sec:kgenerational}, proving Theorem \ref{thm:xjyjwjzj} and Corollary \ref{cor:kgenerationalexists} as well as results about the limiting behavior of $|kA|$ and $|kA-kA|$ as $k$ grows, improving earlier results of Nathanson \cite{Na1}. We conclude in Section \ref{sec:simulcomparison} with a proof of Theorem \ref{thm:simulcomparison}.


\section{Generalized MSTD Sets}\label{sec:proofthmgeneralizedMSTD}

The goal of this section is, given $k\in\mathbb{N}$ and integers with $s_{1}+d_{1}=s_{2}+d_{2}=k$ and $\{s_{1},d_1\}\neq \{s_{2},d_{2}\}$, to explicitly construct a set $A$ such that $\left|s_{1}A-d_{1}A\right|=\left|s_{2}A-d_{2}A\right|+1$. The existence of these sets is the key ingredient in the proof of Theorem \ref{thm:xjyjwjzj}; the proof of that theorem requires us to generalize these constructions slightly, and then modify the arguments of Martin and O'Bryant to obtain a positive percentage by showing a positive percentage of middles may be added to our set. The reader uninterested in the technical construction should skim the sketch of the method in Remark \ref{rek:sketchmainproof} and then continue in \S\ref{sec:pospercent}.

Let $\ell=2k+1$, $r=2k+2$, and consider the sets
\bea\label{eq:LReq}
L &\ = \ & \{0,1,3,4,\ldots,k-1,k,k+1,2k+1\}\nn
 & = & \{0,\ell-2k,l-2k-2,\ell-2k-3,\ldots,\ell-k-1,\ell-k,\ell\}\nn
 & = & [0,\ell]\backslash\left(\{2\}\cup[k+2,2k]\right)\nn
 & =& [0,\ell]\backslash\left(\{2\}\cup[\ell-k+1,\ell-1]\right)\nn
R & = & \{0,1,2,4,5,\ldots,k,k+1,k+2,2k+2\}\nn
 & =& [0,r]\backslash\left(\{3\}\cup[k+3,2k+1]\right).\eea

We begin with a technical lemma. This lemma states that for any $x,y\in\mathbb{N}$,
the basic structure of $xL+yR$ is the same as that of the original
sets. Basically, $xL+yR$ is always missing the first $k$ elements
below the maximum, as well as the singleton element $2k-1$ away from
the maximum. Even more, it is missing no other elements.

\begin{lem}\label{lem:xLyRproperties}
For all $x,y\in\mathbb{N}$, \be xL+yR \ = \ [0,x\ell+yr]\backslash\left([x\ell+yr-k+1,x\ell+yr-1]\cup\{x\ell+yr-2k+1\}\right). \ee
\end{lem}

\begin{proof}
The proof is by double induction, first on $x$, then on $y$. As the proof of the base case $x=y=1$ follows by a simple computation, we now assume the result for $xL+yR$ and prove it for $xL+(y+1)R$.

We are interested in
\be \left([0,x\ell+yr] \backslash \left([x\ell+yr-k+1,x\ell+yr-1]\cup\{x\ell+yr-2k+1\}\right)\right)+R.\ee
We prove that this set contains the proper elements in several steps:\\

\noindent\emph{Claim 1:} $[0,x\ell+yr]\subset xL+(y+1)R$.

\noindent\emph{Proof:} Clearly $x\ell+yr-2k+1\in xL+(y+1)R$, since $x\ell+yr-2k\in xL+yR$
and $1\in R$. Furthermore, $[x\ell+yr-k+1,x\ell+yr-1]\subset xL+(y+1)R$, since
$x\ell+yr-k-4\in xL+yR$, $x\ell+yr-k-1\in xL+yR$, and $[4,k+2]\subset R$.\\

\noindent\emph{Claim 2:} $x\ell+(y+1)r-2k+1=x\ell+yr+3\notin xL+(y+1)R$.

\noindent\emph{Proof:} This is equivalent to showing that $x\ell+yr+3-(xL+yR)\cap R=\emptyset$.
This is true as $x\ell+yr+3-(xL+yR)\cap(\mathbb{N}\cup\{0\})=\{3\}$, and $3\notin R$.\\

\noindent\emph{Claim 3:} $[x\ell+(y+1)r-k,x\ell+(y+1)r-1]\cap\left(xL+(y+1)R\right)=\emptyset$.

\noindent\emph{Proof:} This is the same as showing that $[x\ell+yr+k+2,k\ell+yr+2k+1]\cap(xL+(y+1)R)=\emptyset$.
I.e., we want to show that $\max(xL+(y+1)R)-a\notin xL+(y+1)R$ for
every $1\leq a\leq k$. This is true because $\max(xL+yR)-a\notin xL+yR$
and $\max(R)-a\notin R$ for every $1\leq a\leq k$. Therefore the
same will be true of $xL+yR+R=xL+(y+1)R$.\\

\noindent\emph{Claim 4:} All other elements in $[x\ell+yr,x\ell+(y+1)r]$ are in
$xL+(y+1)R$.

\noindent\emph{Proof:} This is true because each of those elements can be written as $x\ell+yr+c$
for some $c\in R$.\\

We have proved the inductive step for $y$; we omit the proof
of the inductive step for $x$, since it is almost exactly the same
as the above proof.
\end{proof}

With the technical lemma proved, we can construct a set as in Theorem \ref{thm:generalizedMSTD}.

\begin{thm}\label{thm:onegeneralizedMSTDset}
Suppose $k\in\mathbb{N}$, and $s_{1}+d_{1}=s_{2}+d_{2}=k$. Further
suppose that $\{s_{1},d_1\}\neq \{s_{2},d_{2}\}$. There exists a set $A$ such
that $\left|s_{1}A-d_{1}A\right|=\left|s_{2}A-d_{2}A\right|+1$.
\end{thm}

For example, the set $$A=\{0,1,3,4,5,9,33,34,35,50,54,55,56,58,59,60\}$$ has the property that $$\left|A+A+A+A\right| > \left|A+A+A-A\right|.$$

\begin{proof} Because $\left|xA-yA\right|=\left|yA-xA\right|$, we can assume that
$s_{1}\geq d_{1}$ and $s_{2}\geq d_{2}$. Therefore we have either
$s_{1}>s_{2}\geq d_{2}>d_{1}$, or $s_{2}>s_{1}\geq d_{1}>d_{2}$.
We first treat the case when $s_{1}>s_{2}$.\\

\noindent \emph{Case 1: $s_1 > s_2$:} Take $L,R,\ell,r$ as in construction from Lemma \ref{lem:xLyRproperties}, and choose $n>4(kr-2k+1)$. Define \bea
M & \ = \  & [kr-2k+1-d_{1},kn-(kr-2k+1-d_{1})] \nonumber\\
A & =  & L\cup M\cup(n-R). \eea

To prove this, we first show that the middle of $s_1A-d_1A$ is full, and then we examine the fringes. We have $[(kr-2k+1-d_1)-d_1n,s_1n-(kr-2k+1-d_1)]\subset s_1A-d_1A$. To prove this, note
that $M$ is sufficiently large such that $(M+L)\cup(M+n-R)$ is the
entire interval $[\min(M),n+\max(M)]$. Therefore, $[\min(M),n+\max(M)]\subset A+A$. Similarly, we get that $[\min(M)-n,\max(M)]\subset A-A$. The same idea shows that $(M+L)\cup(M+n-R)$
is sufficiently large such that \[
[\min(M),2n+\max(M)]\subset\left(M+L+L\right)\cup\left(M+M+L\right)\cup\cdots\cup\left(M+n-R+n-R\right).\]
By induction, $s_1A-d_1A$ will contain $[\min(M),(k-1)n+\max(M)]$.

We first look at the left fringe of $s_{1}A-d_{1}A$, this is (up
to translation) $sL+dR\cap[0,kr-2k-d_{1}]$. Note that $kr-2k-d_{1}=s_{1}\ell+d_{1}r-2k-d_{1}+s_{1}$.
Therefore, using Lemma \ref{lem:xLyRproperties}, we get that \be s_{1}L+d_{1}R\cap[0,s_{1}\ell+d_{1}r-2k-d_{1}+s_{1}] \  = \ [0,kr-2k-d_{1}]\backslash\{s_{1}\ell+d_{1}r-2k+1\} \ee (this is because $s_{1}>d_{1}$). Next we look at the right fringe.
This is (up to translation and a minus sign) $d_{1}L+s_{1}R\cap[0,kr-2k-d_{1}]$,
which is the same as \be d_{1}L+s_{1}R\cap[0,d_{1}\ell+s_{1}r-2k-d_{1}+d_{1}] \ = \ d_{1}L+s_{1}R\cap[0,d_{1}\ell+s_{1}r-2k].\ee
Using Lemma \ref{lem:xLyRproperties}, this is just $[0,d_{1}\ell+s_{1}r-2k]$ (i.e.,
the entire interval). Therefore $s_{1}A-d_{1}A$ is missing one element.

Next, we look at the left fringe of $s_{2}A-d_{2}A$. Once again,
this is (up to translation) $s_{2}L+d_{2}R\cap[0,kr-sk+1-d_{1}]$.
This can be rewritten as $s_{2}L+d_{2}R\cap[0,s_{2}\ell+d_{2}r-2k-d_{1}+s_{2}]$.
Since we have that $s_{2}>d_{1}$, we get \be
s_{2}L+d_{2}R\cap[0,s_{2}\ell+d_{2}r-2k-d_{1}+s_{2}] \ = \ [0,s_{2}\ell+d_{2}r-2k-d_{1}+s_{2}]\backslash\{s_{2}\ell+d_{2}r-2k+1\}.\ee
Therefore the left fringe is missing one element. Now we look at the
right fringe. This is (up to translation and a minus sign) $d_{2}L+s_{2}R\cap[0,kr-2k+1-d_{1}]$.
This is the same as $d_{2}L+s_{2}R\cap[0,d_{2}\ell+s_{2}r-2k+1-d_{1}+d_{2}]$.
Now, because we have $d_{2}>d_{1}$, this intersection is $[0,d_{2}\ell+s_{2}r-2k+1-d_{1}+d_{2}]\backslash\{d_{2}\ell+s_{2}r-2k+1\}$.
Therefore, the right fringe is missing one element. This means that
$kn=\left|s_{1}A-d_{1}A\right|>\left|s_{2}-d_{2}A\right|=kn-1$.\\

\noindent \emph{Case 2: $s_2 > s_1$:} As $s_2 > s_1$ we have $d_1 > d_2$. Define
\bea
M & \ =  \  & [kr-2k+1-s_{1},n-(kr-2k+1-s_{1})]\nonumber\\
A&=&L\cup M\cup(n-R).\eea
We have \bea
s_{1}L+d_{1}R\cap[0,kr-2k+1-s_{1}] &\ =\ & s_{1}L+d_{1}R\cap[0,s_{1}\ell+d_{1}r-2k+1+s_{1}-s_{1}]\nn
 &\ =\ & [0,s_{1}\ell+d_{1}r-2k+1],\eea so the left fringe is missing no elements. Furthermore
\bea
d_{1}L+s_{1}R\cap[0,kr-2k+1-s_{1}] &\ =\ & d_{1}L+s_{1}R\cap[0,d_{1}\ell+s_{1}r-2k+1-s_{1}+d_{1}]\nn
 &\ =\ & [0,d_{1}\ell+s_{1}r-2k+1-s_{1}+d_{1}]. \eea The last step is true because $s_{1}\geq d_{1}$. Therefore, $s_{1}A-d_{1}A$
misses no elements.

Next, we look at:\bea
s_{2}L+d_{2}R\cap[0,kr-2k+1-s_{1}] &\ =\ & s_{2}L+d_{2}R\cap[0,s_{2}\ell+d_{2}r-2k+1-s_{1}+s_{2}]\nn
 & = & [0,s_{2}\ell+d_{2}r-2k+1-s_{1}+s_{2}]\backslash\{s_{2}\ell+d_{2}r-2k+1\}, \nonumber\\ \eea
which is true because $s_{2}>s_{1}$. This is enough to show that $\left|s_{1}A-d_{1}A\right|>\left|s_{2}A-d_{2}A\right|$,
but we will go slightly further and show that $\left|s_{1}A-d_{1}A\right|=\left|s_{2}A-d_{2}A\right|+1$.
To do this, we look at the right fringe of $s_{2}A-d_{2}A$. As $s_1 > d_2$, we have\bea
d_{2}L+s_{2}R\cap[0,kr-2k+1-s_{1}] &\ = \ & d_{2}L+s_{2}R\cap[0,d_{2}\ell+s_{2}r-2k+1-s_{1}+d_{2}]\nn
 & = &[0,d_{2}\ell+s_{2}r-2k+1-s_{1}+d_{2}], \eea which completes the proof. \end{proof}

Although it doesn't matter for our current purposes, the following lemma will be important later. Each of the sets constructed above is sum-difference
balanced both before and after the critical point. More formally,
we have the following.

\begin{lem}\label{lem:balancedbeforek}
In all the sets $A$ defined in the proof of Theorem \ref{thm:onegeneralizedMSTDset}, \be \left|s_{1}A-d_{1}A\right|\ = \ \left|s_{2}A-d_{2}A\right|\ee
for any $s_{1}+d_{1}=s_{2}+d_{2}$ such that $s_{1}+d_{1}\neq k$.\end{lem}

\begin{proof}
In every one of the constructions above, $sA-dA$ contains all possible
numbers whenever $s+d>k$, so it only remains to show this fact when
$s+d<k$.

This essentially follows from Lemma \ref{lem:xLyRproperties}. Both $s_{1}A-d_{1}A$
and $s_{2}A-d_{2}A$ contain the same middle (up to translation),
so it is enough to analyze the fringes. When $s_{1}+d_{1}<k$,
these fringes do not intersect the middle, so it suffices
to show that \be \left|s_{1}L+d_{1}R\right|+\left|d_{1}L+s_{1}R\right|\ = \ \left|s_{2}L+d_{2}R\right|+\left|d_{2}L+s_{2}R\right|.
\ee Using Lemma \ref{lem:xLyRproperties}, we know that $\left|sL+dR\right|=s\ell+dr-k$.
Therefore, it is enough to show that \be s_{1}\ell+d_{1}r-k+d_{1}\ell+s_{1}r-k \  = \ s_{2}\ell+d_{1}r-k+d_{2}\ell+s_{2}r-k.\ee
This equation is the same as \be (s_{1}+d_{1})(\ell+r)-2k \ = \ (s_{2}+d_{2})(\ell+r)-2k; \ee as $s_1+d_1 = s_2+d_2$, the above is true, which completes the proof.
\end{proof}

\section{Positive Percentages}\label{sec:pospercent}

We now give a proof of Part 2 of Theorem \ref{thm:generalizedMSTD}.

\begin{lem}\label{lem:pospercent}
Suppose there exists a finite set $A\subseteq\mathbb{Z}$ such that $\left|s_{1}A-d_{1}A\right|>\left|s_{2}A-d_{2}A\right|$, where $s_{1}+d_{1}=s_{2}+d_{2}$. Further suppose that $s_{1}\geq2$. Then \be \liminf_{n\to\infty} \frac{\#\{B\subseteq [0,n-1] ; \, \left|s_1B-d_1B\right|>\left|s_2B-d_2B\right|\}}{2^n} \ > \ 0; \ee in other words, a positive percentage of subsets have this structure.
\end{lem}
\begin{rek}  Note that the assumption $s_1\geq2$ only rules out the case $s_{1}=d_{1}=1$,
since we can always replace $A$ with $-A$ without affecting the cardinalities. This case has already been dealt with in detail in \cite{MO}.
\end{rek}

\begin{proof}
By translation, we can assume that $A\subseteq [0,n-1]$, with $0,n-1\in A$.

Choose some $m\geq4(s_{1}+d_{1})n$, and define \bea
L &\ =\ & A\cup[(s_{1}+d_{1})n,2(s_{1}+d_{1})n-1] \nonumber\\ U&=&[m-2(s_{1}+d_{1})n,m-(s_{1}+d_{1})n-1]\cup(A+(m-n)).\eea Informally, our fringes consist of a copy of $A$ at the far end,
then an interval of size $(s_{1}+d_{1})n$ which is located $(s_{1}+d_{1})n$
away from the edge.

Furthermore, define $l=u=2(s_{1}+d_{1})n$. Next, we note three things:
\begin{enumerate}
\item $[\frac{l}{2},2l-2]\subset L+L$
\item $[2m-2u,2m-\frac{u}{2}-1]\subset U+U$
\item $[m-u,m+l-2]\subset L+U$.
\end{enumerate}
Each of these claims follows from $[\frac{l}{2},l-1]\subset L$ and $[m-u,m-\frac{u}{2}-1]\subset U$,
as well as the fact that $0\in L$, and $m-1\in U$

Next, suppose that $B\subset[0,m-1]$ is a set with fringes $L,U$.
Based on Proposition 8 of \cite{MO}, the probability
that\be
[2l-1,m-u-1]\cup[m+l-1,2m-2u-1]\ \subseteq\ B+B\ee is at least \be 1-6(2^{-\left|L\right|}+2^{-\left|U\right|})\ > \ 1-6(2^{-(s_{1}+d_{1})n}+2^{-(s_{1}+d_{1})n})\ = \ 1-6\cdot2^{-(s_{1}+d_{1})n+1}\ = \ c. \ee

Therefore, if $B$ is a set as above, then with positive probability
(that is independent of $m$), \be
\left[\frac{l}{2},2m-\frac{u}{2}-1\right]\ \subset\ B+B.\ee Essentially, we have chosen the fringes of $B$ such that with a positive
probability that is independent of $m$, the entire middle (here middle
means everything besides the $(s_{1}+d_{1})n$ elements on each side)
of $B+B$ will be full.
However, this means that the entire middle
of $s_{1}B-d_{1}B$ will also be full. Therefore, it only remains
to check the fringes of $s_{1}B-d_{1}B$. Each of
these fringes is just a copy of $s_{1}A-d_{1}A$.
Therefore, $s_{1}B-d_{1}B$ consists of a copy on $s_{1}A-d_{1}A$
on each fringe, and everything in between.

To show that $\left|s_{1}B-d_{1}B\right|>\left|s_{2}B-d_{2}B\right|$,
it is sufficient to note that for the exact same reasons, the fringes
of $s_{2}B-d_{2}B$ will just be copies of $s_{2}A-d_{2}A$. Therefore,
since $s_{1}B-d_{1}B$ contains strictly more elements on the fringe,
as well as everything not on the fringe, it must have more elements
that $s_{2}B-d_{2}B$.

As for the probability, we have made $4(s_{1}+d_{1})n$ choices for
the fringes of $B$, and making sure the middle is full accounts for a factor of $c$. So the probability that $\left|s_{1}B-d_{1}B\right|>\left|s_{2}B-d_{2}B\right|$
is at least $c2^{-4(s_{1}+d_{1})n}$.
\end{proof}

\section{Technical Constructions}\label{sec:technicalconstructions}

\subsection{Multiple Fringes}

In order to prove Theorems \ref{thm:xjyjwjzj} and \ref{thm:simulcomparison}, we first construct a very well behaved set. Then, in Section \S\ref{sec:kgenerational} we will use the base expansion method to create a set that combines many different copies of the below set.

\begin{lem}\label{lem:superniceset}
Suppose $k\in\mathbb{N}$, and $s,d\in\mathbb{N}\cup\{0\}$ such that
$s+d=k$ and $s\geq d$. There exists a set $A\subset\mathbb{N}\cup\{0\}$
such that if $s'+d'=k$, $s'\neq s$, and $s'\geq d'$, then $\left|sA-dA\right|=\left|s'A-d'A\right|+1$.\end{lem}
\begin{proof}
Have $L,R$ as in \eqref{eq:LReq}:
\bea
L&\ =\ & \{0,1,3,4,\ldots,k-1,k,k+1,2k+1\}\nonumber\\
&=& \{2k+1\} \cup [0,k+1] \setminus \{2\} \nonumber\\
R&=&\{0,1,2,4,5,\ldots,k,k+1,k+2,2k+2\} \nonumber\\
&=& \{2k+2\} \cup [0,k+2] \setminus \{3\}.\eea Set $\ell=2k+1$ and $r=2k+2$ (just as before).

Before we give the main proof, there are two exceptional cases to
consider. We have already proved the case where $d=0$. If $s=d$,
then choose $n>2\left(kr-2k+1-d\right)$ and take \be A\ = \ L\cup[kr-2k+1-d,n-(kr-2k+1-d)]\cup(n-R).\ee
For this set, $sA-dA$ misses no elements, and $s'A-d'A$
misses one element for any choice of $s',d'$ that satisfies
the above. The proof of this statement is essentially the same as
found in the above proofs.

Now we assume that $s,d\geq1$ and that $s>d$. Set
\bea
A& \ =  \  & L \cup (L+kr-2k+1-d) \cup [2kr-4k+2-d-s,n-(2kr-4k+2-d-s)] \nonumber\\ & & \ \ \ \cup (n-(2kr-2k+1-d)-R) \cup (n-R).\eea
Essentially, $A$ consists of an outer fringe, and inner fringe, and
a full middle. Both the outer fringe and the inner fringe have the
same structure (they are both made up of $L$ and $R$). For simplicity,
we write this as \be
A\ = \ L_{1}\cup L_{2}\cup M\cup(n-R_{2})\cup(n-R_{1}),\ee
where $L_{1}=L$, $L_{2}=L+kr-2k+1-d$, $R_{1}=R$, $R_{2}=R+kr-2k+1-d$,
and $M=[2kr-4k+2-d-s,n-(2kr+2-d-s)]$.

Note first that because $n$ is sufficiently large, $sA-dA$ and $s'A-d'A$
will contian the entire middle (the logic for this is the same as
above). Further note that the fringes of $sA-dA$ are \be (sL_{1}-d(n-R_{1}))\cup(L_{2}+(s-1)L_{1}-d(n-R_{1}))\ee
and \be (s(n-R_{1})-dL_{1})\cup((n-R_{2})+(s-1)(n-R_{1})-dL_{1}.\ee This
is because all other sums/differences fall in the large and full middle.
As usual, we will translate these sets (and possibly multiply by $-1$),
and look at \be (sL_{1}+dR_{1})\cup(L_{2}+(s-1)L_{1}+dR_{1}) \ {\rm and}\ (sR_{1}+dL_{1})\cup(R_{2}+(s-1)R_{1}+dL_{1}).\ee
We analyze each of these four fringes one at a time.
\begin{enumerate}
\item $sL_{1}+dR_{1}$

First note that $L_{2}+(s-1)L_{1}+dR_{1}$ contains the interval $[kr-k+1-s,kr-s]=[s\ell+dr-k+1,s\ell+dr-1]$.
This means that of the potential missing elements in $sL_{1}+dR_{1}$,
all except for $kr-2k+1$ can be found in $L_{2}+(s-1)L_{1}+dR_{1}$.
Essentially, we are interested in $sL_{1}+dR_{1}\cap[0,kr-2k-d]$.
This is $sL_{1}+dR_{1}\cap[0,s\ell+dr-2k-d+s]$, which is just $[0,kr-2k-d]\backslash\{s\ell+dr-2k+1\}$
(because $s>d$). Therefore the outer left fringe is missing one element.

\item $L_{2}+(s-1)L_{1}+dR_{1}$

Part of this set will intersect with the full middle, so we are really
only interested in $L_{2}+(s-1)L_{1}+dR_{1}\cap[kr-2k+1-d,2kr-4k+1-d-s]$.
After translation, this is the same as $sL_{1}+dR_{1}\cap[0,kr-2k-s]$.
This is $sL_{1}+dR_{1}\cap[0,s\ell+dr-2k]$, which is just $[0,s\ell+dr-2k]$.
Therefore the inner left fringe is missing no elements.

\item $sR_{1}+dL_{1}$

Similar to the above case, this fringe intersects with $R_{2}+(s-1)R_{1}+dL_{1}$.
Therefore, we are only interested in $sR_{1}+dL_{1}\cap[0,kr-2k-d]$.
This is the same as $sR_{1}+dL_{1}\cap[0,sr+d\ell-2k]$, which is
just $[0,sr+d\ell-2k${]}. Therefore the outer right fringe is missing
no elements.

\item $R_{2}+(s-1)R_{1}+dL_{1}$

Because of the intersection with the middle, we are only interested
in $R_{2}+(s-1)R_{1}+dL_{1}\cap[kr-2k+1-d,2kr-4k+1-d-s]$. After translation,
this is just $sR_{1}+dL_{1}\cap[0,kr-2k-s]=sR_{1}+dL_{1}\cap[0,sr+d\ell-2k-s+d]$.
Since $s>d$, this is just the entire interval $[0,kr-2k-s]$. Therefore
the inner right fringe is missing no elements.

\end{enumerate}
Next, we run through the same analysis with $s'A-d'A$. We split this
into two cases. First, if $s'>s$, then:
\begin{enumerate}
\item $s'L_{1}+d'R_{1}$

Just as above, we are interested in $s'L_{1}+d'R_{1}\cap[0,kr-2k-d]$.
This is $s'L_{1}+d'R_{1}\cap[0,s'\ell+d'r-2k-d+s']=[0,kr-2k-d]\backslash\{s'\ell+dr-2k+1\}$.
Therefore the outer left fringe is missing one element.

\item $L_{2}+(s'-1)L_{1}+d'R_{1}$

Similar to above, this is the same thing as $s'L_{1}+d'R_{1}\cap[0,kr-2k-s]$.
This is just $s'L_{1}+d'R_{1}\cap[0,s'\ell+d'r-2k-s+s']$. Since $s'>s$,
this is $[0,s'\ell+d'r-2k-s+s']\backslash\{s'\ell+d'r-2k+1\}$.

\end{enumerate}
Therefore $s'A-d'A$ is missing at least two elements. Only slightly
more work shows that the set is missing exactly two elements, which
means that $\left|sA-dA\right|=\left|s'A-d'A\right|+1$.

Next, we assume that $s>s'$. In this case we have $s>s'\geq d'>d$.
If we perform the same analysis as above, we will find that $sL_{1}+dR_{1}$
and $sR_{1}+dL_{1}$ are both missing one element. Therefore we get
that $\left|sA-dA\right|=\left|s'A-d'A\right|+1$.\end{proof}
\begin{cor}
Suppose that $k\in\mathbb{N}$, and $s,d\in\mathbb{N}\cup\{0\}$ such
that $s+d=k$. Then there exists a set $A\subseteq\mathbb{N}\cup\{0\}$
such that if $s'+d'=k$, $\{s,d\}\neq\{s',d'\}$, then $\left|sA-dA\right|=\left|s'A-d'A\right|+1$.\end{cor}
\begin{proof}
This follows from the above if we just note that $\left|sA-dA\right|=\left|-(sA-dA)\right|$.
\end{proof}


\subsection{Base expansion}

We end this section with a quick proof of the base expansion method for creating new sets. Base expansion allows us to use multiple copies of the well-behaved sets constructed in Lemma \ref{lem:superniceset} to create the sets in Theorems \ref{thm:xjyjwjzj} and \ref{thm:simulcomparison}.

\begin{lem}\label{lem:basexpansionlem} Fix a positive integer $k$. Let $A,B\subset\mathbb{N}\cup\{0\}$ and choose $m>k\cdot \max(A)$.
Let $C=A+m\cdot B$ (where $m\cdot B$ is the usual scalar multiplication). Then $\left|sC-dC\right|=\left|sA-dA\right|\cdot\left|sB-dB\right|$
whenever $s+d\leq k$.\end{lem}
\begin{proof}
Note that each element of $A+mB$ can be written uniquely as $a+mb$
for some $a\in A$, $b\in B$. This is true because if $a_{1}+mb_{1}=a_{2}+mb_{2}$,
then $a_{1}-a_{2}=m(b_{2}-b_{1})$. Because we chose $m$ sufficiently
large, this is only possible when $b_{1}=b_{2}$, in which case $a_{1}-a_{2}=0$.
Therefore $\left|C\right|=\left|A\right|\left|B\right|$.

Furthermore, each element of $C\pm C$ can be written uniquely as
$a'\pm mb'$, where $a'\in A\pm A$ and $b'\in B\pm B$. As proof,
assume $a_{1}\pm mb_{1}=a_{2}\pm mb_{2}$ for some $a_{1,}a_{2}\in A\pm A$,
and $b_{1},b_{2}\in B\pm B$. This means $a_{1}-a_{2}=\mp m(b_{2}-b_{1})$,
and this is only possible when $ $$a_{1}=a_{2}$, and $b_{1}=b_{2}$.
Therefore, $\left|C\pm C\right|=\left|A\pm A\right|\left|B\pm B\right|$.
A similar proof shows this fact for any $s+d\leq k$.
\end{proof}
In fact, base expansion works in more generality:

\begin{lem} Fix a positive integer $k$. Say that $A_{1},\ldots,A_{k}\subset\mathbb{N}\cup\{0\}$. Choose some
$m>k\cdot \max(\{a;a\in A_{k}\text{ for some }k\})$. Let $C=A_{1}+m\cdot A_{2}+\cdots+m^{k-1}\cdot A_{k}$ (where $m\cdot A_j$ is the usual scalar multiplication).
Then $\left|sC-dC\right|=\prod_{j=1}^{k}\left|sA_{j}-dA_{j}\right|$
whenever $s+d\leq k$.\end{lem}
\begin{proof}
This can be proved using induction and the previous lemma.
\end{proof}

\section{$k$-Generational Sets}\label{sec:kgenerational}

\subsection{Proof of Theorem \ref{thm:xjyjwjzj}}

We now have the tools required to prove our results about chains.

\begin{proof}[Theorem \ref{thm:xjyjwjzj}]
For each $j$, choose a set $A_{j}$ such that $\left|x_{j}A_{j}-y_{j}A_{j}\right|$ $>$ $\left|w_{j}A_{j}-z_{j}A_{j}\right|$,
and $|s_{1}A_{j}$ $-$ $d_{1}A_{j}|$ $=$ $|s_{2}A_{j}$ $-$ $d_{2}A_{j}|$
whenever $s_{1}+d_{1}=s_{2}+d_{2}\neq j$. We know such a set exists,
because of Theorem \ref{thm:generalizedMSTD} and Lemma \ref{lem:balancedbeforek}. Next, choose some $m>k\cdot\max(\{a\in A_{j};\,1\leq j\leq k\})$.
Define $A=A_{1}+mA_{2}+m^{2}A_{3}+\cdots+m^{k-1}A_{k}$. We have that
for each $2\leq j\leq k$ \begin{align}
\left|x_{j}A-y_{j}A\right| &\ =\ \prod_{i=1}^{k}\left|x_{j}A_{i}-y_{j}A_{i}\right|\nn
 &\ =\ \left|x_{j}A_{j}-y_{j}A_{j}\right|\cdot\prod_{i\neq j}\left|x_{j}A_{i}-y_{j}A_{i}\right|\nn
 &\ =\ \left|x_{j}A_{j}-y_{j}A_{j}\right|\cdot\prod_{i\neq j}\left|w_{j}A_{i}-z_{j}A_{i}\right|\nn
 &\ >\ \left|w_{j}A_{j}-z_{j}A_{j}\right|\cdot\prod_{i\neq j}\left|w_{j}A_{i}-z_{j}A_{i}\right|\nn
 &\ =\ \left|w_{j}A-z_{j}A\right|.\end{align}
\end{proof}

Most of Corollary \ref{cor:kgenerationalexists} now follows automatically. The existence of a $k$-generational set is proven by the above theorem, and proving that a positive percentage of sets have this property only requires a slight modification of the work done in \S\ref{sec:pospercent}. It only remains to that no set can be $k$-generational for all $k$ by analyzing the limiting behavior of $|kA|$ and $|kA-kA|$.

\subsection{Limiting behavior of $|kA|$ and $|kA-kA|$}

Before proving Corollary \ref{cor:kgenerationalexists}(3), we first prove two useful lemmas.

\begin{lem}\label{lem:behaviorkA} Let $A=\{a_1,a_2,\ldots,a_m\}\subset[0,n-1]$ be a set of integers where $a_1<a_2<\cdots<a_m$ and let $s=\gcd(a_1,a_2,\ldots,a_m)$. Then there exists an integer $N$ such that for $k\ge N$ we have $|kA|=\frac{k(a_m-a_1)}{s}-C$ where $C$ is a constant and $N$ is bounded above by $\frac{a_m-a_1}{s}$.
\end{lem}

\begin{proof} It suffices to show that a set of the form $\{0, a_1,\ldots,a_m\}$ with $\gcd(a_1,\ldots,a_m)=1$ has the claimed properties (because of translating and rescaling).

Let $A=\{0, a_1,\ldots,a_m\}$, with $\gcd(a_1,\ldots,a_m)=1$. We first show that in $a_1A$ (which is the sum of $a_1$ copies of $A$) there are elements of each congruence class of $a_1$. Consider the set $B=\{0,a'_2,\ldots,a'_m\}$ where $a'_i=a_i \bmod a_1$. Clearly we also have $\gcd(a'_2,\ldots,a'_m)=1$. Thus they generate the entire set $[0,a_1]$. It is clear that the largest number of times required to add $B$ to itself is $a_1$ since the greatest order of any element in the set is $a_1$. This proves the claim.

Now consider $a_mA$, in particular we consider the set $L=a_mA\cap[0,a_1a_m]$. We show that $kA\cap[0,a_1a_m]=L$ for $k\ge a_m$. This is because $a_1$ is the smallest element in the set $A$, so elements that are less than $a_1a_m$ can be written as $\sum_{i=1}^m\epsilon_ia_i$ where $\sum_i^m\epsilon_i\le a_m$. We call $L$ the stabilized left fringe of $A$.

We can apply the same idea to the set $a_m-A$ and show that the right fringe $R=kA\cap[(k-1)a_m,ka_m]$ is also stabilized (meaning that $ka_m-R$ stays the same for all $k\ge a_m$). Now we just need to show that for $k\ge a_m$ we have $kA\backslash(L\cup R)$ is completely filled. This can be shown by induction. With all the congruence classes of $a_1$, by brute force we can show that the middle part of $a_mA$ is completely filled. This serves as the base case of the induction. If $kA\backslash(L\cup R)$ is completely filled then $kA$ contains the interval $[a_1a_m,(k-a_m+a_{m-1})a_m]$. If we add $a_m$ to this interval we will get the interval $[(k-a_m+a_{m-1})a_m,(k+1-a_m+a_{m-1})a_m]$. So in $(k+1)A$, we will have a completely filled middle $[a_1a_m,(k+1-a_m+a_{m-1})a_m]$. This completes the proof.
\end{proof}

\begin{lem}\label{cor:|kA-kA|beats|kA+kA|}
Let $A=\{a_1,a_2,\ldots,a_m\}\subset[0,n-1]$ be a set of integers where $a_1<a_2<\cdots<a_m$ and let $s=\gcd(a_1,a_2,\ldots,a_m)$. Then there exists an integer $N$ such that for $k\ge N$ we have $|kA-kA|\ge|kA+kA|$ and $N$ is bounded above by $\frac{2(a_m-a_1)}{s}$.
\end{lem}

\begin{proof}
Note that $kA\subset kA-kA$. This means that if $cA+cA$
has stable fringes and a full middle, then $2cA-2cA$ will contain
all those fringe elements (and maybe more) as well as the full middle.
Therefore, if we choose $N=2c$, then for any $k\geq N$, $|kA-kA|\geq|kA+kA|$
\end{proof}

Corollary \ref{cor:kgenerationalexists}(3) now follows immediately; in other words, no set can be $k$-generational for all $k$. This significantly improves an earlier result of Nathanson \cite{Na1}, who proved that $kA$ stabilizes by $k\ge a^2m$, where $a$ is the largest element of $A$ and $m$ is the largest gap between elements of $A$.

\section{Simultaneous Comparison}\label{sec:simulcomparison}

In this section we prove that any ordering for a simultaneous comparison happens.

\begin{proof}[Proof of Theorem \ref{thm:simulcomparison}]
We repeatedly use base expansion. For each $1\leq j\leq n$,
choose $A_{j}$ such that $\left|s_{j}A_{j}-d_{j}A_{j}\right|=\left|sA-dA\right|+1$
for every $s\neq\pm s_{j}$. Next, choose an $m>k\cdot\ \max(\{a;a\in A_{j}\text{ for some }1\leq j\leq n\})$.
Let \bea
A&=&A_{1}+mA_{2}+m^{2}A_{2}+\cdots+\underbrace{m^{\frac{j(j-1)}{2}}A_{j}+\cdots+m^{\frac{j(j-1)}{2}+j-1}A_{j}}_{j\text{ times }}+\cdots+m^{\frac{n(n-1)}{2}+n-1}A_{n}.\nonumber\\ \eea More simply, $A$ is made of $j$ copies of each $A_{j}$. Arguing as before (such as in Lemma \ref{lem:basexpansionlem}), we find
\begin{eqnarray}
\left|s_{j}A-d_{j}A\right| \ =\ \prod_{i}\left|s_{j}A_{i}-d_{j}A_{i}\right|^{i}.\end{eqnarray}

Now, we have that $\left|s_{j}A_{i}-d_{j}A_{i}\right|=\left|s_{\ell}A_{i}-d_{\ell}A_{i}\right|$
whenever $\ell,j\neq i$. Furthermore, we have that $\left|s_{i}A_{i}-d_{i}A_{i}\right|=\left|s_{j}A_{i}-d_{j}A_{i}\right|+1$
whenever $i\neq j$. Therefore, if we choose $j>\ell$, we obtain
\begin{eqnarray}
\left|s_{j}A-d_{j}A\right| & =&\prod_{i}\left|s_{j}A_{i}-d_{j}A_{i}\right|^{i} \nn
 & =&\left|s_{j}A_{j}-d_{j}A_{j}\right|^{j}\cdot\prod_{i\neq j}\left|s_{j}A_{i}-d_{j}A_{i}\right|^{i} \nn
 & =&\left(\left|s_{\ell}A_{j}-d_{\ell}A_{j}\right|+1\right)^{j}\cdot\left(\left|s_{\ell}A_{\ell}-d_{\ell}A_{\ell}\right|-1\right)^{\ell}\prod_{i\neq j,\ell}\left|s_{\ell}A_{i}-d_{\ell}A_{i}\right|^{i} \nn
 & =&\left|s_{\ell}A_{\ell}-d_{\ell}A_{\ell}\right|^{j}\left|s_{\ell}A_{j}-d_{\ell}A_{j}\right|^{\ell}\prod_{i\neq j,\ell}\left|s_{\ell}A_{i}-d_{\ell}A_{i}\right|^{i} \nn
 & >&\left|s_{\ell}A_{\ell}-d_{\ell}A_{\ell}\right|^{\ell}\left|s_{\ell}A_{j}-d_{\ell}A_{j}\right|^{j}\prod_{i\neq j,\ell}\left|s_{\ell}A_{i}-d_{\ell}A_{i}\right|^{i} \nn
 &\ = \ &\prod_{i}\left|s_{\ell}A_{i}-d_{\ell}A_{i}\right|^{i}=\left|s_{\ell}A-d_{\ell}A\right|.
 \end{eqnarray}

Informally, we have chosen the $A_{i}$ such that $\left|s_{i}A_{i}-d_{i}A_{i}\right|$
is larger than all other possible combinations of sums and differences.
Then we made $\left|s_{2}A-d_{2}A\right|>\left|s_{1}A-d_{1}A\right|$
by having more copies of $A_{2}$ than of $A_{1}$. Similarly, we
made $\left|s_{3}A-d_{3}A\right|>\left|s_{2}A-d_{2}A\right|$ by having
more copies of $A_{3}$ than of $A_{2}$. Following this process,
we constructed a set $A$ with the desired properties.

We have found an $A$ such that $\left|s_{n}A-d_{n}A\right|>\cdots>\left|s_{1}A-d_{1}A\right|$, completing the proof.
\end{proof}

\appendix


\section{Arbitrary Differences}\label{sec:arbitrarydifference}

In this section we prove Theorem \ref{thm:arbitrarydifference}. Let
\begin{eqnarray}\label{eqn: arbitraryA}
A\ =\ L\cup [16km -2m +1, n-(16km-2m+1)] \cup (n-R)
\end{eqnarray}
where
\begin{eqnarray}
L &=& [0, 4m]\cup [5m+1,6m]\cup \{8m\} \nonumber\\
R &=& (L+ m/k)\cup [0, m/k-1].
\end{eqnarray}
Note that the fringes $L,R$ of this $A$ are generalizations of the original fringes in \eqref{eq:LReq}. For example, this new $L$ is obtained from the original $L$ by extending the first gap of the original $L$ to have length $m$. Also, note that this $R$ is $L$ shifted down by $m/k$, with the front filled in; this generalizes the original $R$ in \eqref{eq:LReq}, where $R$ is $L$ shifted down only by $1$.

We modify this $A$ in several steps, each step bringing us closer to the full generality of Theorem \ref{thm:arbitrarydifference}.
We first show that the above $A$ has the property that $|kA+kA| = 2kn+1 - m$ and $|kA-kA| = 2kn+1 - 2m$ so that $|kA+kA| - |kA-kA| = m$.  Note that this fringe only works if $m$ is a multiple of $k$ since $R$ is shifted by $m/k$. In the second step, we fix this to allow $m$ that is not a multiple of $k$ by partially filling in the first gap of $L, R$. In the third step, we construct $A$ such that $|kA+kA| = 2kn+1 - m$ and $|kA-kA| = 2kn+1 - \ell$ for any $\ell \le 2m$ by extending the middle interval $[16km -2m +1, n-(16km-2m+1)]$ of $A$. In the last step, we get the full theorem for general $a,b,c,d$ by changing how much $R$ is shifted from $L$.
\\ \\
\emph{Step 1: m is a multiple of k.}\\

We first prove that if $m$ is a multiple of $k$, the above $A$ has $|kA+kA| = 2kn +1-m,  |kA-kA| = 2kn+1 - 2m$. To find $|kA+kA|, |kA-kA|$, we need to study the fringes of $kA+kA, kA-kA$. We will use Lemma \ref{lem:xLyRproperties}, which says that for any $x,y$ that
 $xL +yR$ is a translation of $L$ with the front filled in. In general, note that if $R$ is shifted down from $L$ by $d$, we have that $xL+yR$ ends at $x(8m) + y(8m + d) = (x+y)(8m) + yd$ and if $x+y$ is fixed, the result depends only on $y$ and $d$.
Hence as in Figure \ref{fig:set4s}, the left fringe $kL+kL$ of $kA+kA$ moves slower than the right fringe $kR+kR$. Therefore the right fringe of $kA+kA$ reaches the middle before the left fringe of $kA+kA$, resulting in some missing elements in the left fringe but no missing elements in the right fringe. By Figure \ref{fig:set2s2d}, the fringes $kL+kR$ of $kA-kA$ each have some missing elements since $kL+kR$ also moves slower than $kR+kR$.

To be precise, the left fringe $kL + kL$ of $kA+kA$ is
\begin{eqnarray}
kL+kL &\ = \ & [0, 16km - 4m ]\cup [16km -3m +1,16km- 2m]\cup \{16km\}\nonumber\\
&=& (L+ 16km - 8m) \cup [0, 16km-8m -1].
\end{eqnarray}
Note by \eqref{eqn: arbitraryA} that the middle of $kA+kA$ on the left side starts at $16km-2m+1$. Therefore,  $kA+kA$ is missing the $m$ elements in $[16km-4m+1, 16km -3m]$ in its left fringe.

The right fringe of $kA+kA$ is $2kn - (kR+kR)$ and so after reflection, we only need to study $kR+kR$, which is
\begin{eqnarray}
kR+kR\ =\  [0, 16km - 2m] \cup [16km-m +1, 16km] \cup \{16km +2m\}.
\end{eqnarray}
Again by \eqref{eqn: arbitraryA} note that the middle of $kA+kA$ on the right side starts at $2kn -(16km-2m+1)$, which is $16km - 2m +1$ after reflection. This covers the missing elements of $kR+kR$ and so $kA+kA$ has no missing elements in its right fringe.

Since the middle of $kA+kA$ is filled in, $kA+kA$ has all elements except for the $m$ missing elements in its left fringe and so $|kA+kA| = 2kn+1 - m$.

Now we need to study the fringes of $kA-kA$. Note that $kA-kA$ is symmetric so the left and right fringes are the same. The left fringe of $kA-kA$ is $kL-k(n-R) = kL +kR - kn$. After translation, we can study $kL+kR$, which is
\begin{eqnarray}
kL + kR\ =\ [0, 16km -3m]\cup [16km -2m+1 , 16km-m]\cup \{16km +m \}.
\end{eqnarray}
After translation, the middle of $kA-kA$ starts on the left side at $16km - 2m +1$ as before.
Therefore, the middle covers the first gap $[16km-m+1,16km+m-1]$  in $kL+kR$  but not the second gap $[16km-3m+1,16km-2m]$, which has $m$ elements.
Therefore, the left fringe of $kA-kA$ has $m$ missing elements. By symmetry, the right fringe of $kA-kA$ also has $m$ missing elements. Since the middle of $kA-kA$ is filled in, $kA-kA$ has all elements except for $2m$ elements and so $|kA-kA| = 2kn+1 - 2m$.

Finally, we note that it is sufficient to take $n$ such that $n - 2(16km-2m+1)> 16m$. We make $n$ large enough so that the middle of $A$ has size at least $16m$, the size of the original fringes $L, R$. In fact, we just need that the middle of $kA+kA, kA-kA$ has enough elements to cover the second gap of the $kL+kL, kR+kR$, and $kL+kR$.
\\ \\
\emph{Step 2: m is not a multiple of k.}
\\

To do the case when $m$ is not a multiple of $k$, we use the same fringes as before but partially fill in their gaps.
Let $m'$ be the smallest multiple of $k$ that is greater than or equal to $m$.
By \eqref{eqn: arbitraryA} and Step 1, we can construct $A'$ such that $|kA'+kA'| = 2kn+1 - m'$ and $|kA'-kA'| = 2kn+1 - 2m'$.
That is, the left fringe of $A'$ is
\begin{eqnarray}
L'\ =\ [0, 4m']\cup [5m'+1,6m']\cup \{8m'\}
\end{eqnarray}
so that $L'$ is defined like the original $L$ but for $m'$ instead of $m$. Now we note that since the middle of $A'$ starts at $16km' - 2m' + 1$, the first gap of $L'$ accounts for all the missing elements of $kL' + kL'$ and $kL'+ kR'$.
In fact, a copy of $L' \cap [4m' + 1, 5m']$ appears identically in the left fringe of $kA'+kA'$ and $kA'-kA'$.
Therefore, we can fill in the first $m'-m$ elements of the first gap of $L'$ by considering
\begin{eqnarray}
L''\ =\ L' \cup [4m'+1,4m'+(m'-m)].
\end{eqnarray}
and do the same to construct $R''$ from $R'$. Then $kL'' + kL''$ will have only $m$ missing elements since $kL' + kL'$ has $m'$ missing elements and we filled in $m'-m$ elements.  Also note that $kR'' +  kR''$ has no missing elements since $kR' + kR'$ did not have any missing elements. Thus, if we construct $A$ from $L''$ and $R''$, we have $|kA+kA| = 2n+1 - m$. Note that for this construction, we can fill in
\emph{any} $m'-m$ elements of the first gap of $L'$, not necessarily the first $m'-m$ elements.

Similarly, $kL''+kR''$ now misses only $m$ elements since it also has a copy of $L'' \cap [4m' + 1, 5m']$. Therefore $kA-kA$ has $m$ missing elements in each fringe and so $|kA-kA| = 2kn+1 - 2m$.
\\ \\
\emph{Step 3: Arbitrary $m, \ell \le 2m$.}
\\

Now we further modify $A$ so that for any $m$ and $\ell \le m$, we have $|kA+kA|= kn+1-m$ and $|kA-kA|= kn+1 -2\ell$. Note that again we must do the cases when $m$ is multiple of $k$ and when $m$ is not a multiple separately. However, we only do the case where $m$ is a multiple of $k$ since from Step 2, it is clear how to extend to other case.

In particular, we will modify $A$ by extending the middle section in both directions by $m-\ell$. Therefore the middle of $kA+kA$ now starts at
$16km - 2m +1 - (m-\ell)$. Recall that the missing elements in $kL +kL$ occur only from the first gap $[16km-4m+1, 16km -3m]$. Since $\ell \ge 0$, we have $16km - 2m + 1-(m-\ell) \ge 16km - 3m +1$ and so $kL +kL$ is still missing $m$ elements. As before, $kR+kR$ has no missing elements and so we still have$|kA+kA| = 2kn +1 -m$.

On the other hand, $kL+kR$ has fewer missing elements than it usually would. Note that now the middle of $kA - kA$ also starts at $16km - 2m +1 - (m-\ell) = 16km - 3m + \ell +1$.
Since the missing elements in $kL +kR$ occur only from the first gap $[16km-3m+1, 16km - 2m]$ of $kL+kR$, then $kL + kR$ has only the missing $\ell$ elements $[16km-3m+1, 16km - 3m + \ell]$. Therefore, we get that $kA-kA$ is missing only $\ell$ elements in each fringe and so $|kA-kA| = 2kn+1 -2\ell$.

Note that we cannot do better than having $|kA-kA| = 2kn +1 -2\ell$ with $\ell \le m$ with this approach. Shortening the middle does not help since although it increases the number of missing elements in $kA-kA$, it also increases the number of missing elements in $kA+kA$.
\\ \\
\emph{Step 4: Arbitrary $a,b,c,d$.}
\\

Finally, we modify $A$ to prove the desired theorem for arbitrary $a,b,c,d$. In particular, we will modify $A$ by changing how much $R$ is shifted from $L$. This changes the speed at which the right fringe approaches the middle. We adjust the speed so that the right fringe of $aA-bA$ has no missing elements while all the other fringes still have some missing elements.

We again make some simplifying assumptions. We will only construct $A$ such that $|aA-bA| = qn + 1- m$ and $|cA- dA| = qn +1 - 2m$ since we can use the methods from Step 3 to extend to the case with
$|aA-bA| = qn + 1- m, |cA-dA| = qn +1 - \ell$, where $\ell \le 2m$.
Since $|aA-bA| =|bA-aA|$, we can assume $a > b$ and $c> d$. Furthermore, since $a+b= c+d =q$ and $a$ is the maximal element, we have that $a>c$ and $b<d$.
We first assume that $c\ne d$ and then discuss how to do case when $c=d$; note that in the case $c=d$, we must have $\ell$ be even.
 We must also break up the proof into the case when $c-d\le d-b$ and when $c-d > d-b$. We will only do the case when $c-d\le d-b$ and then discuss how to do the other case. Finally, we must consider separately the case when $m$ is a multiple of $c-d$ and when $m$ is not; we will only do the former since the latter follows as in Step 2.

We now construct $A$ such that $|aA-bA| = qn + 1- m$ and $|cA- dA| = qn +1 - 2m$, with $c \ne d$ and $m$ a multiple of $c-d$.

We first let $\Delta = m/(c-d)$ and
\begin{eqnarray}
&& L\ =\ [0, 2\Delta (a-b)]\cup [2\Delta(a-b) + \Delta (c-d) +1, 3\Delta(a-b)]\cup\{4\Delta(a-b)\}\nonumber\\
&& R\ =\ (L + \Delta) \cup [0, \Delta-1].
\end{eqnarray}
These fringes are similar to the fringes in \eqref{eqn: arbitraryA} except that the middle block of $L,R$ has a different size and $R$ is shifted from $L$ by a different amount. Also let
\begin{equation}
 A\ =\  L \cup
 [ 4\Delta(a-b)(a+b) + \Delta a - 2\Delta( a-b)+1, n - (4\Delta(a-b)(a+b) + \Delta a - 2\Delta( a-b)+1)]
  \cup (n- R).
\end{equation}
The middle is chosen to start at
$4\Delta(a-b)(a+b) + \Delta a - 2\Delta( a-b)+1$, which is 1 after the end of the first block of $bL+aR$,  the right fringe of $aA-bA$.

We first study $aA-bA$. The left fringe of $aA-bA$ is $aL -(b(n-R))$, which is $aL +bR$ after translation. The maximum element of $aL +bR$ is
\begin{eqnarray}
4\Delta(a-b)a + (4\Delta (a-b)+\Delta)b\ =\ 4\Delta(a-b)(a+b) + \Delta b
\end{eqnarray}
and the pattern to the left of the maximum element is exactly the same as in $L$ and $R$. That is,
\begin{eqnarray}
&&aL+bR  \nonumber\\
&&\ = \ [0, 4\Delta(a-b)(a+b) + \Delta b - 2\Delta( a-b)] \nonumber\\
&&\ \ \ \cup [ 4\Delta(a-b)(a+b) + \Delta b - 2\Delta(a-b) + \Delta (c-d)+1, 4\Delta(a-b)(a+b)+\Delta b - \Delta(a-b)]\nonumber\\
&&\ \ \ \cup  \{ 4\Delta(a-b)(a+b)+ \Delta b\}.
\end{eqnarray}
Since the middle of $A$ starts at
\begin{equation}
4\Delta(a-b)(a+b) + \Delta a - 2\Delta( a-b)+1\ =\ 4\Delta(a-b)(a+b) + \Delta b - \Delta(a-b)+1,
\end{equation}
we see that $aL+bR$ is missing the $\Delta (c-d) = m$ elements
\begin{equation}
[4\Delta(a-b)(a+b) + \Delta b - 2\Delta( a-b)+1,  4\Delta(a-b)(a+b) + \Delta b - 2\Delta(a-b) + \Delta (c-d)],
\end{equation}
which are all the elements in its first gap.

Now we must consider the right fringe $a(n-R) - bL$ of $aA - bA$, which $bL +aR$ after reflection. Note that
\begin{eqnarray}
&&bL+aR\ =\ aL +bR + \Delta (a-b) \nonumber\\
&&\ = \ [0, 4\Delta(a-b)(a+b) + \Delta a - 2\Delta( a-b)]\nonumber\\
&&\ \ \ \cup [ 4\Delta(a-b)(a+b) + \Delta a - 2\Delta(a-b) + \Delta (c-d)+1, 4\Delta(a-b)(a+b)+\Delta a - \Delta(a-b)]\nonumber\\
&&\ \ \ \cup\{ 4\Delta(a-b)(a+b)+ \Delta a\},
\end{eqnarray}
with the first equality since $R$ is $L$ shifted down by $\Delta$. Note that $bL + aR$ is not missing any elements since the middle of $A$ starts at
\begin{eqnarray}
4\Delta(a-b)(a+b) + \Delta a - 2\Delta( a-b)+1,
\end{eqnarray}
which is exactly where the first gap of $bL+aR$ starts. Therefore $|aA-bA| = 2kn+1 - m$ since $aA-bB$ is missing only $m$ elements in its left fringe.

Now we will consider $cA-dA$. Its left fringe is $cL - d(n-R)$, which is $cL+dR$ after translation. Then as before,
\begin{eqnarray}
&&cL+dR\  =\ aL +bR + \Delta (d-b) \nonumber\\
&&\ =\ [0, 4\Delta(a-b)(a+b) + \Delta d - 2\Delta( a-b)] \nonumber\\
&&\ \ \ \cup[ 4\Delta(a-b)(a+b) + \Delta d - 2\Delta(a-b) + \Delta (c-d)+1, 4\Delta(a-b)(a+b)+\Delta d - \Delta(a-b)]\nonumber\\
&&\ \ \ \cup\{ 4\Delta(a-b)(a+b)+ \Delta d\}.
\end{eqnarray}
Then $cL+dR$ is only missing the $\Delta (c-d) =m$ elements
\begin{equation}
[4\Delta(a-b)(a+b) + \Delta d - 2\Delta( a-b)+1,  4\Delta(a-b)(a+b) + \Delta d - 2\Delta(a-b) + \Delta (c-d)],
\end{equation}
which are all the elements in its first gap. This is because the middle of $A$ starts in the middle of the second block of $cL+dR$ since
\begin{equation}
4\Delta(a-b)(a+b) + \Delta d -2\Delta(a-b) + \Delta(c-d) +1\ \le\ 4\Delta(a-b)(a+b) + \Delta a -2\Delta(a-b)
\end{equation}
as $c< a$ and
\begin{equation}
4\Delta(a-b)(a+b) + \Delta a -2\Delta(a-b) \ \le\ 4\Delta(a-b)(a+b)+\Delta d - \Delta(a-b)
\end{equation}
as $b > d$.

Similarly, the right fringe of $cA-dA$ is $c(n-R) - dL$, which is $dL + cR$. Then
\begin{eqnarray}
&&dL+cR \ =\ aL +bR + \Delta (c-b)\nonumber\\
&&\ = \ [0, 4\Delta(a-b)(a+b) + \Delta c - 2\Delta( a-b)]\nonumber\\
&&\ \ \ \cup[ 4\Delta(a-b)(a+b) + \Delta c - 2\Delta(a-b) + \Delta (c-d)+1, 4\Delta(a-b)(a+b)+\Delta c - \Delta(a-b)]\nonumber\\
&&\ \ \ \cup\{ 4\Delta(a-b)(a+b)+ \Delta c\},
\end{eqnarray}
and as before, $dL +cR$ is only missing the $m$ elements
\begin{equation}
[4\Delta(a-b)(a+b) + \Delta c - 2\Delta( a-b)+1,  4\Delta(a-b)(a+b) + \Delta c - 2\Delta(a-b) + \Delta (c-d)],
\end{equation}
which are all the elements in its first gap. This is because the middle of $A$ starts in the middle of the second block of $dL+cR$ since
\begin{equation}
4\Delta(a-b)(a+b) + \Delta c -2\Delta(a-b) + \Delta(c-d) +1\ \le\ 4\Delta(a-b)(a+b) + \Delta a -2\Delta(a-b)
\end{equation}
and
\begin{equation}
 4\Delta(a-b)(a+b) + \Delta a -2\Delta(a-b)\ \le\ 4\Delta(a-b)(a+b)+\Delta c - \Delta(a-b).
\end{equation}
To verify the first inequality, we note that $2c-d < a$ holds since $a+b = c+d$ and $c-b < d-b$ in this case.
The second inequality follows from $b \le c$.
Therefore, $cA-dA$ is missing $m$ elements in each fringe and so $|cA-dA| = 2kn+1 - 2m$.
\\

To do the case $c-d > d-b$, we need to change the fringes slightly. However, the only real difference occurs when we extend the middle to get $|cA-dA|= 2kn+1 - \ell$, where $\ell \le 2m$, as in Step 3. We do this by first extending the middle one element at a time (to decrease $\ell$ one element at a time). However, at a certain point we need to extend the middle by adding a whole block;  at this point extending one element does not change the value of $|cA-dA|$ and so we just extend by a whole block. Afterwards, we continue extending the middle one element at a time as before.

Finally, we note that the case when $c=d$ is similar to the result achieved in Step 3, except that now the left fringe $aL + bR$ of $aA-bA$ is closer to the middle; therefore we need to make the middle shorter so that the middle misses the first gap in the left fringe of $aL+bR$. This completes the proof of Theorem \ref{thm:arbitrarydifference}. \hfill $\Box$


\ \\
\end{document}